\documentclass[a4paper,11pt]{article}

\usepackage{amsmath,amssymb,amsthm}
\usepackage{amssymb,latexsym, bbm,comment}

\renewcommand{\theequation}                            
       {\mbox{\arabic{section}.\arabic{equation}}}

\newcommand{\origsetminus}{} \let\origsetminus=\setminus           
\renewcommand{\setminus}{\!\origsetminus\!}

{\theoremstyle{plain}
\newtheorem{definition}{Definition}[section]
\newtheorem{lemma}[definition]{Lemma}
\newtheorem{theorem}[definition]{Theorem}
\newtheorem{corollary}[definition]{Corollary}

}
{\theoremstyle{definition}
\newtheorem{example}[definition]{Example}
\newtheorem{remark}[definition]{Remark}
}

\renewcommand{\mathbb}{\mathbbm}                     
\renewcommand{\epsilon}{\varepsilon}                 
\renewcommand{\phi}{\varphi}
\renewcommand{\theta}{\vartheta}
\renewcommand{\le}{\leqslant}
\renewcommand{\ge}{\geqslant}

\newcommand{\origfoo}{} \let\origfoo=\sqrt           
\renewcommand{\sqrt}[1]{\origfoo{#1}\;}

\newcommand{\abs}[1]{\left\lvert #1 \right\rvert}    
\newcommand{\norm}[1]{\left\lVert #1 \right\rVert}   
\DeclareMathOperator{\F}{{\cal F}}                   
\DeclareMathOperator{\R}{{\mathbb R}}                
\DeclareMathOperator{\Rp}{{\mathbb R}_+}             
\DeclareMathOperator{\C}{{\mathbb C}}                
\DeclareMathOperator{\N}{{\mathbb N}}                
\DeclareMathOperator{\Id}{ Id}                        

\DeclareMathOperator{\Cov}{\text{Cov}}
\newcommand{\A}{{\mathcal A}}
\DeclareMathOperator{\Borel}{{\mathcal B}}
\newcommand{\scapro}[2]{\langle #1,#2\rangle}       
\newcommand{\scaproh}[2]{\left[ #1,#2\right]_H}       

\DeclareMathOperator{\Z}{{\mathcal Z}} \DeclareMathOperator{\Cc}{{\mathcal C}}
\DeclareMathOperator{\Var}{\text{Var}}

\newcounter{zahl}

\title{Cylindrical Wiener processes}

\author{ Markus Riedle\footnote{ markus.riedle@manchester.ac.uk}\\
   School of Mathematics\\
   The University of Manchester \\
   M13 9PL, UK
}

\begin{document}
\maketitle

\begin{abstract}
  In this work cylindrical Wiener processes on Banach spaces are defined by means of cylindrical
  stochastic processes, which are a well considered mathematical object. This approach
  allows a definition which is a simple straightforward extension of
  the real-valued situation. We apply this definition to introduce
 a stochastic integral with respect to cylindrical Wiener processes.
Again, this definition is a straightforward extension of the real-valued situation which
results now in simple conditions on the integrand. In particular, we do not have to put
any geometric constraints on the Banach space under consideration. Finally, we
   relate this integral to well-known stochastic integrals in literature.
\end{abstract}

%

\tableofcontents

%


%

\section{Introduction}

Cylindrical Wiener processes appear in a huge variety of models in
infinite dimensional spaces as a source of random noise or random
perturbation. Almost in the same amount as models with cylindrical
Wiener processes one can find different definitions of cylindrical
Wiener processes in literature. Most of these definitions suffer
from the fact that they do not generalise  comprehensibly the
real-valued definition to the infinite dimensional situation.

In this note cylindrical Wiener processes on a Banach space are
introduced by virtue of the core mathematical object which underlies
all these definitions but which is most often not mentioned: a {\em
cylindrical stochastic process}. A cylindrical stochastic process is
a generalised stochastic process whose distribution at a fixed time
defines only a finite countably additive set function on the Banach
space. These finite countably additive set functions are called {\em
cylindrical measures}. We give a very transparent definition of a
weakly cylindrical Wiener process as a {\em cylindrical stochastic
process which is Wiener}. Our approach has the side-effect that the
appearance of the word {\em cylindrical} is given a reason.

This definition of a weakly cylindrical Wiener process is a
straightforward extension of the real-valued situation but it is
immediately seen to be too general in order to be analytically
tractable.  An obvious request is that the covariance operator of
the associated Gaussian cylindrical measures exists and has the
analogue properties as in the case of ordinary Gaussian measures on
infinite-dimensional spaces. This leads to a second definition of a
strongly cylindrical Wiener process.

For strongly cylindrical Wiener processes we derive a representation
by a series with independent real-valued Wiener processes. On the
other hand, we see, that by such a series a strongly cylindrical
Wiener process can be constructed.

The obvious question when is a cylindrical Wiener process actually a Wiener process in
the ordinary sense can be answered easily thanks to our approach by the self-suggesting
answer: if and only if the underlying cylindrical measure extends to an infinite
countably additive set function, i.e. a measure.

Utilising furthermore the approach by cylindrical measures  we
define a stochastic integral with respect to cylindrical Wiener
processes. Again, this definition is a straightforward extension of
the real-valued situation which results now in simple conditions on
the integrand. In particular, we do not have to put any geometric
constraints on the Banach space under consideration. The cylindrical
approach yields that the distribution of the integral is  a
cylindrical measure. We finish with two corollaries giving
conditions such that the cylindrical distribution of the stochastic
integral extends to a probability measure. These results relate our
integral to other well-known integrals in literature.

To summarise, this article introduces two major ideas:
\begin{itemize}
\item
A cylindrical Wiener process is defined by a straightforward
extension of the real-valued situation and the requirement of having
a nice covariance operator. It can be seen that most of the existing
definitions in literature have the same purpose of guaranteeing the
existence of an analytically tractable covariance operator. Thus,
our definition unifies the existing definitions and respects the
core mathematical object underlying the idea of a cylindrical Wiener
process.
\item
Describing a random dynamic in an infinite dimensional space by an
ordinary stochastic process requires the knowledge that it is a {\em
real} infinite dimensional phenomena, i.e. that there exits a
probability measure on the state space. Whilst describing the
dynamic by a cylindrical stochastic process it is sufficient to know
only the finite dimensional dynamic under the application of all
linear bounded functionals. Our introduced stochastic integral
allows the development of such a theory of cylindrical stochastic
dynamical systems and has the advantage that no constraints are put
on the underlying space.
\end{itemize}

We do not claim that we accomplish very new mathematics in this
work. But the innovation might be seen by relating several
mathematical objects which results in a straightforward definition
of a cylindrical Wiener process and its integral. Even these
relations might be well known to some mathematicians but they do not
seem to be accessible in a written form.

Our work relies on several ingredients from the theory of
cylindrical and ordinary measures on infinite dimensional spaces.
Based on the monographs Bogachev \cite{Bogachev98} and Vakhaniya et
al \cite{Vaketal} we give an introduction to this subject. The
section on $\gamma$-radonifying operators is based on the notes by
Jan van Neerven \cite{JanSeminar}.  Cylindrical Wiener processes in
Banach or Hilbert spaces and their integral are treated for example
in the monographs Da Prato and Zabcyzk \cite{DaPrato92}, Kallianpur
\cite{Kallianpur} and Metivier and Pellaumail \cite{Metivier80}. In
van Gaans \cite{vanGaans} the series representation of the
cylindrical Wiener process is used to define a stochastic integral
in Hilbert spaces and in Berman and Root \cite{BerRo83} an approach
similar to ours is introduced. The fundamental observation in this
work that not every Gaussian cylindrical measure has a nice
covariance operator was pointed out to me the first time by Dave
Applebaum.

\section{Preliminaries}

Throughout this notes let $U$ be a separable
 Banach space with dual $U^\ast$. The dual pairing is denoted by $\scapro{u}{u^\ast}$ for $u\in U$ and
 $u^\ast\in U^\ast$. If $V$ is another Banach space then $L(U,V)$ is the space of all linear, bounded
 operators from $U$ to $V$ equipped with the operator norm $\norm{\cdot}_{U\to V}$.

The Borel $\sigma$-algebra is denoted by  $\Borel(U)$. Let $\Gamma$ be a subset of $U^\ast$. Sets of the form
 \begin{align*}
\Z(u_1^\ast,\dots ,u_n^\ast,B):= \{u\in U:\, (\scapro{u}{u_1^\ast},\cdots,
 \scapro{u}{u_n^\ast})\in B\},
\end{align*}
where $u_1^\ast,\dots, u_n^\ast\in \Gamma$ and $B\in \Borel(\R^n)$ are called {\em cylindrical sets } or {\em
cylinder with respect to $(U,\Gamma)$}. The set of all cylindrical sets is denoted by $Z(U,\Gamma)$, which
turns out to be an algebra. The generated $\sigma$-algebra is denoted by $\Cc(U,\Gamma)$ and it is called
{\em cylindrical $\sigma$-algebra with respect to $(U,\Gamma)$}. If $\Gamma=U^\ast$ we write
$\Cc(U):=\Cc(U,\Gamma)$. If $U$ is separable then both the Borel $\Borel(U)$ and the cylindrical
$\sigma$-algebra $\Cc(U)$ coincide.

A function $\mu:\Cc(U)\to [0,\infty]$ is called {\em cylindrical
measure on $\Cc(U)$}, if for each finite subset $\Gamma\subseteq
U^\ast$ the restriction of $\mu$ on the $\sigma$-algebra
$\Cc(U,\Gamma)$ is a measure. A cylindrical measure is called finite
if $\mu(U)<\infty$.

For every function $f:U\to\R$ which is measurable with respect to
$\Cc(U,\Gamma)$ for a finite subset $\Gamma\subseteq U^\ast$ the
integral $\int f(u)\,\mu(du)$ is well defined as a real-valued
Lebesgue integral if it exists. In particular, the characteristic
function $\phi_\mu:U^\ast\to\C$ of a finite cylindrical measure
$\mu$ is defined by
\begin{align*}
 \phi_{\mu}(u^\ast):=\int e^{i\scapro{u}{u^\ast}}\,\mu(du)\qquad\text{for all }u^\ast\in
 U^\ast.
\end{align*}

In contrast to measures on infinite dimensional spaces there is an analogue of  Bochner's Theorem for
cylindrical measures:
\begin{theorem}\label{th.bochner}
A function $\phi:U^\ast\to\C$ is a characteristic function of a
cylindrical measure on $U$ if and only if
\begin{enumerate}
\item[{\rm (a)}] $\phi(0)=0$;
\item[{\rm (b)}] $\phi$ is positive definite;
\item[{\rm (c)}] the restriction of $\phi$ to every finite
dimensional subset $\Gamma\subseteq U^\ast$ is continuous with respect to the norm topology.
\end{enumerate}
\end{theorem}
For a finite set $\{u_1^\ast,\dots, u_n^\ast\}\subseteq U^\ast$ a
cylindrical measure $\mu$ defines by
\begin{align*}
\mu_{u_1^\ast,\dots, u_n^\ast}:\Borel(\R^n)\to [0,\infty],\qquad \mu_{u_1^\ast,\dots,
u_n^\ast}(B):=\mu\big(\left\{u\in U:\, (\scapro{u}{u_1^\ast},\dots,
\scapro{u}{u_n^\ast})\in B\right\}\big)
\end{align*}
a measure on $\Borel(\R^n)$. We call $\mu_{u_1^\ast,\dots,
u_n^\ast}$ {\em the image of the measure $\mu$ under the  mapping
$u\mapsto (\scapro{u}{u_1^\ast},\dots, \scapro{u}{u_n^\ast})$}.
Consequently, we have for the characteristic function
$\phi_{\mu_{u_1^\ast,\dots, u_n^\ast}}$ of
$\mu_{u_1^\ast,\dots,u_n^\ast}$ that
\begin{align*}
\phi_{\mu_{u_1^\ast,\dots, u_n^\ast}}(\beta_1,\dots, \beta_n)
=\phi_\mu(\beta_1u_1^\ast+\dots + \beta_nu_n^\ast)
\end{align*}
for all $\beta_1,\dots , \beta_n\in\R$.

Cylindrical measures are described uniquely by their characteristic
functions and therefore by their one-dimensional distributions
$\mu_{u^\ast}$ for $u^\ast\in U^\ast$.


\section{Gaussian cylindrical measures}

A measure $\mu$ on $\Borel(\R)$ is called Gaussian with
mean $m\in\R$ and variance $\sigma^2\ge 0$ if either $\mu=\delta_m$ and $\sigma^2=0$ or
it has the density
\begin{align*}
 f:\R\to\Rp,\qquad f(s)=\tfrac{1}{\sqrt{2\pi\sigma^2}} \exp\left(-\tfrac{1}{2\sigma^2}(s-m)^2\right).
\end{align*}
In case of a multidimensional or an infinite dimensional space $U$ a
measure $\mu$ on $\Borel(U)$ is called Gaussian if the image
measures $\mu_{u^\ast}$ are Gaussian for all $u^\ast\in U^\ast$.
Gaussian cylindrical measures are defined analogously but due to
some reasons explained below we have to distinguish between two
cases: weakly and strongly Gaussian.
\begin{definition}
A cylindrical measure $\mu$ on $\Cc(U)$ is called  {\em weakly Gaussian} if $\mu_{u^\ast}$ is Gaussian on
$\Borel(\R)$ for every $u^\ast\in U^\ast$.
\end{definition}

Because of well-known properties of Gaussian measures in finite dimensional Euclidean spaces a cylindrical measure
$\mu$ is weakly Gaussian if and only if $\mu_{u_1^\ast, \dots, u_n^\ast}$ is a Gaussian measure on $\Borel(\R^n)$
for all $u_1^\ast, \dots, u^\ast_n\in U^\ast$ and all $n\in\N$.

\begin{theorem}\label{th.Gausschar}
Let $\mu$ be a weakly Gaussian cylindrical measure on $\Cc(U)$. Then its characteristic function $\phi_{\mu}$ is of the
form
\begin{align}\label{eq.charGauss}
\phi_{\mu}:U^\ast \to\C,\qquad \phi_{\mu}(u^\ast) =\exp\left( i m(u^\ast) -\tfrac{1}{2} s(u^\ast)\right),
\end{align}
where the functions $m:U^\ast \to \R$ and $ s:U^\ast \to \Rp$ are given by
\begin{align*}
 m(u^\ast)=\int_{U} \scapro{u}{u^\ast}\,\mu(du), \qquad s(u^\ast)=\int_{U} \scapro{u}{u^\ast}^2 \mu(du) - (m(u^\ast))^2.
\end{align*}
Conversely, if $\mu$ is a cylindrical measure  with characteristic function of the form
\begin{align*}
\phi_{\mu}:U^\ast \to\C,\qquad \phi_{\mu}(u^\ast) =\exp\left( i m(u^\ast) -\tfrac{1}{2} s(u^\ast)\right),
\end{align*}
for a linear functional $m:U^\ast \to \R$ and a quadratic form $ s:U^\ast \to \Rp$, then
$\mu$ is a weakly  Gaussian cylindrical measure.
\end{theorem}
\begin{proof}
  Follows from \cite[Prop.IV.2.7]{Vaketal}, see also \cite[p.393]{Vaketal}.
\end{proof}

\begin{example}
 Let $H$ be a separable Hilbert space. Then the function
\begin{align*}
  \phi:H\to\C,\qquad \phi(u)=\exp(-\tfrac{1}{2}\norm{u}^2_H)
\end{align*}
satisfies the condition of Theorem \ref{th.Gausschar} and therefore there exists a weakly Gaussian
cylindrical measure $\gamma$ with characteristic function $\phi$. We call this cylindrical measure {\em
standard Gaussian cylindrical measure on $H$}. If $H$ is infinite dimensional the cylindrical measure
$\gamma$ is not a measure, see \cite[Cor.2.3.2]{Bogachev98}.

Note, that this example might be not applicable for a Banach space
$U$ because then $x\mapsto \norm{x}^2_U$ need not to be a quadratic
form.
\end{example}

For a weakly Gaussian cylindrical measure $\mu$   one defines for $u^\ast,v^\ast\in U^\ast$:
\begin{align*}
  r(u^\ast,v^\ast):=\int_U \scapro{u}{u^\ast}\scapro{u}{v^\ast}\,\mu(du) -
   \int_U \scapro{u}{u^\ast}\,\mu(du) \int_U \scapro{u}{v^\ast}\,\mu(du).
\end{align*}
These integrals exist as $\mu$ is a Gaussian measure on the cylindrical $\sigma$-algebra generated by
$u^\ast$ and $v^\ast$. One defines the {\em covariance operator $Q$ of $\mu$} by
\begin{align*}
  Q:U^\ast\to (U^\ast)^\prime,\qquad (Qu^\ast)v^\ast:=r(u^\ast,v^\ast) \qquad\text{for all }v^\ast\in
  U^\ast,
\end{align*}
where $(U^\ast)^\prime$ denotes the algebraic dual of $U^\ast$, i.e. all linear but not necessarily
continuous functionals on $U^\ast$.
 Hence, the characteristic function $\phi_\mu$ of $\mu$ can be written as
\begin{align*}
\phi_{\mu}:U^\ast \to\C,\qquad \phi_{\mu}(u^\ast) =\exp\left( i
m(u^\ast) - (Qu^\ast)u^\ast\right).
\end{align*}
The cylindrical measure $\mu$ is called {\em centered} if $m(u^\ast)=0$ for all $u^\ast\in U^\ast$.

If $\mu$ is a Gaussian measure or more general, a measure of weak order $2$, i.e.
\begin{align*}
  \int_{U}\abs{\scapro{u}{u^\ast}}^2\, \mu(du)<\infty\qquad\text{for all }u^\ast\in U^\ast,
\end{align*}
then the covariance operator $Q$ is defined in the same way as
above. However, in this case it turns out that $Qu^\ast$ is not only
continuous and thus in $U^{\ast\ast}$ but even in $U$ considered as
a subspace of $U^{\ast\ast}$, see \cite[Thm.III.2.1]{Vaketal}. This
is basically due to properties of the Pettis integral in Banach
spaces. For cylindrical measures we have to distinguish this
property and define:
\begin{definition}
A centred weakly Gaussian cylindrical measure $\mu$ on $\Cc(U)$ is called {\em strongly Gaussian} if the
covariance operator $Q:U^\ast\to (U^\ast)^\prime$ is $U$-valued.
\end{definition}

Below Example \ref{ex.weaknotstrong} gives an example of a weakly Gaussian cylindrical measure which is not
strongly. This example can be constructed in every infinite dimensional space in particular  in every Hilbert
space.

Strongly Gaussian cylindrical measures exhibit an other very
important property:
\begin{theorem}\label{th.strongGauss}
For a cylindrical measure $\mu$ on $\Cc(U)$ the following are equivalent:
\begin{enumerate}
\item[{\rm (a)}] $\mu$ is a continuous linear image of the standard Gaussian cylindrical
  measure on a Hilbert space;
\item[{\rm (b)}] there exists a symmetric positive operator $Q:U^\ast\to U$ such that
\begin{align*}
\phi_{\mu}(u^\ast)= \exp\left(-\tfrac{1}{2} \scapro{Qu^\ast}{u^\ast}\right) \qquad\text{for all }u^\ast\in
U^\ast.
\end{align*}
\end{enumerate}
\end{theorem}
\begin{proof}
  See \cite[Prop.VI.3.3]{Vaketal}.
\end{proof}
Theorem \ref{th.strongGauss} provides an example of a weakly Gaussian cylindrical measure which is not
strongly Gaussian:
\begin{example}\label{ex.weaknotstrong}
For a discontinuous linear functional $f:U^\ast\to\R$ define
\begin{align*}
  \phi:U^\ast \to\C,\qquad \phi(u^\ast) =\exp\left(  - \frac{1}{2}(f(u^\ast))^2\right).
\end{align*}
Then $\phi$ is the characteristic function of a weakly Gaussian
cylindrical measure due to Theorem \ref{th.Gausschar} but this
measure can not be strongly Gaussian by Theorem \ref{th.strongGauss}
because every symmetric positive operator $Q:U^\ast\to U$  is
continuous.
\end{example}

\section{Reproducing kernel Hilbert space}

According to Theorem \ref{th.strongGauss} a centred strongly
Gaussian cylindrical measure is the image of the standard Gaussian
cylindrical measure on a Hilbert space $H$ under an operator $F\in
L(H,U)$. In this section we introduce a possible construction of
this Hilbert space $H$ and the operator $F$.

For this purpose we start with a bounded linear operator $Q:U^\ast\to U$, which is positive,
\begin{align*}
  \scapro{Qu^\ast}{u^\ast}\ge 0\qquad\text{for all }u^\ast\in U^\ast,
\end{align*}
and symmetric,
\begin{align*}
  \scapro{Qu^\ast}{v^\ast}=\scapro{Qv^\ast}{u^\ast}\qquad\text{for all }u^\ast,v^\ast\in U^\ast.
\end{align*}
On the range of $Q$ we define a bilinear form by
\begin{align*}
 [Qu^\ast,Qv^\ast]_{H_Q}:=\scapro{Qu^\ast}{v^\ast}.
\end{align*}
It can be seen easily that this defines an inner product
$[\cdot,\cdot]_{H_Q}$. Thus, the range of $Q$ is a pre-Hilbert space
and we denote  by $H_Q$ the real Hilbert space obtained by its
completion with respect to $[ \cdot , \cdot ]_{H_Q}$. This space
will be called the {\em reproducing kernel Hilbert space associated
with $Q$}.

In the following we collect some properties of the reproducing
kernel Hilbert space and its embedding:
\begin{enumerate}
\item[(a)]
The inclusion mapping from the range of $Q$ into $U$ is continuous
with respect to the inner product $[ \cdot, \cdot ]_{H_Q}$. For, we
have
\begin{align*}
 \norm{Qu^\ast}_{H_Q}^2
 =\abs{\scapro{Qu^\ast}{u^\ast}}\le \norm{Q}_{U^\ast\to
 U}\norm{u^\ast}^2,
\end{align*}
which allows us to conclude
\begin{align*}
  \abs{\scapro{Qu^\ast}{v^\ast}}
 &= \abs{[Qu^\ast,Qv^\ast]_{H_Q}}
 \le \norm{Qu^\ast}_{H_Q}\norm{Qv^\ast}_{H_Q}
 \le \norm{Qu^\ast}_{H_Q}\norm{Q}_{U^\ast\to H_Q}\norm{v^\ast}.
\intertext{Therefore, we end up with}
 \norm{Qu^\ast}&=\sup_{\norm{v^\ast}\le 1}\abs{\scapro{Qu^\ast}{v^\ast}}\le \norm{Q}_{U^\ast\to
 H_Q}\norm{Qu^\ast}_{H_Q}.
\end{align*}
Thus the inclusion mapping is continuous on the range of $Q$ and it
extends to a bounded linear operator $i_Q$ from $H_Q$ into $U$.

\item[(b)] The operator $Q$ enjoys the decomposition
\begin{align*}
   Q=i_Q i_Q^\ast.
\end{align*}
For the proof we define $h_{u^\ast}:=Qu^\ast$ for all $u^\ast\in U^\ast$. Then we have
$i_Q(h_{u^\ast})=Qu^\ast$ and
\begin{align*}
 [h_{u^\ast},h_{v^\ast}]_{H_Q}=\scapro{Qu^\ast}{v^\ast}
 =\scapro{i_Q(h_{u^\ast})}{v^\ast}
 =[h_{u^\ast},i_Q^\ast v^\ast]_{H_Q}.
\end{align*}
Because the range of $Q$ is dense in $H_Q$ we arrive at
\begin{align}\label{eq.hQ}
 h_{v^\ast}&=i_Q^\ast v^\ast \qquad\text{for all }v^\ast\in U^\ast
\intertext{which finally leads to}
 Qv^\ast&=i_Q(h_{v^\ast})=i_Q(i_Q^\ast v^\ast)\qquad\text{for all }v^\ast\in U^\ast.\notag
\end{align}
\item[(c)] By \eqref{eq.hQ} it follows immediately that the range of $i_Q^\ast$ is dense in $H_Q$.
\item[(d)] the inclusion mapping $i_Q$ is injective. For, if $i_Qh=0$ for some $h\in H_Q$ it follows that
\begin{align*}
[h,i_Q^\ast u^\ast]_{H_Q}
 =\scapro{i_Q h}{u^\ast}=0\qquad\text{for all }u^\ast\in U^\ast,
\end{align*}
which results in  $h=0$ because of (c).
\item[(e)] If $U$ is separable then $H_Q$ is also separable.
\end{enumerate}


\begin{remark}\label{re.repandimage}
Let $\mu$ be a centred strongly Gaussian cylindrical measure with covariance operator $Q:U^\ast\to U$. Because $Q$ is
positive and symmetric we can associate with $Q$ the reproducing kernel Hilbert space $H_Q$ with the
inclusion mapping $i_Q$ as constructed above. For the image $\gamma\circ i_Q^{-1}$ of the standard
cylindrical measure $\gamma$ on $H_Q$ we calculate
\begin{align*}
\phi_{\gamma\circ i_Q^{-1}}(u^\ast)&=\int_U e^{i\scapro{u}{u^\ast}}\,(\gamma\circ i_Q^{-1})(du)\\
&= \int_{H_Q} e^{i\scapro{h}{i_Q^\ast u^\ast}}\,\gamma(dh)\\
&= \exp\left(-\tfrac{1}{2}\norm{i_Q^\ast u^\ast}_{H_Q}^2\right)\\
&= \exp\left(-\tfrac{1}{2}\scapro{Qu^\ast}{u^\ast}\right).
\end{align*}
Thus, $\mu=\gamma\circ i_{Q}^{-1}$ and we have found one possible Hilbert space and operator satisfying the
condition in Theorem \ref{th.strongGauss}.

But note, that there might exist other Hilbert spaces exhibiting this feature. But the reproducing kernel
Hilbert space is characterised among them by a certain ``minimal property'', see \cite{Bogachev98}.

\end{remark}

\section{$\gamma$-radonifying operators}

This section follows the notes \cite{JanSeminar}.

Let $Q:U^\ast\to U$ be a positive symmetric operator and $H$ the reproducing kernel
Hilbert space with the inclusion mapping $i_Q:H\to U$.  If $U$ is a Hilbert space then it
is a well known result by Mourier (\cite[Thm. IV.2.4]{Vaketal}) that $Q$ is the
covariance operator of a Gaussian measure on $U$ if and only if $Q$ is nuclear or
equivalently if $i_Q$ is Hilbert-Schmidt. By Remark \ref{re.repandimage} it follows that
the cylindrical measure $\gamma\circ i_Q^{-1}$ extends to a Gaussian measure on
$\Borel(U)$ and $Q$ is the covariance operator of this Gaussian measure.

The following definition generalises this property of $i_Q:H\to U$ to define by $Q:=i_Q
i_Q^\ast$ a covariance operator to the case when $U$ is a Banach space:
\begin{definition}
Let $\gamma$ be the standard Gaussian cylindrical measure on a
separable Hilbert space $H$. A linear bounded operator $F:H\to U$ is
called {\em $\gamma$-radonifying} if the cylindrical measure
$\gamma\circ F^{-1}$ extends to a Gaussian measure on $\Borel(U)$.
\end{definition}

\begin{theorem}\label{th.eqradonifying}
Let $\gamma$ be the standard Gaussian cylindrical measure on a
separable Hilbert space $H$ with orthonormal basis $(e_n)_{n\in\N}$
and let $(G_n)_{n\in\N}$ be a sequence of
 independent standard real normal random variables. For $F\in L(H,U)$ the following are equivalent:
 \begin{enumerate}
 \item[{\rm (a)}] $F$ is $\gamma$-radonifying;
 \item[{\rm (b)}] the operator $FF^\ast:U^\ast\to U$ is the covariance operator of a
 Gaussian measure $\mu$ on $\Borel(U)$;
 \item[{\rm (c)}] the series $\displaystyle \sum_{k=1}^\infty G_k Fe_k$ converges a.s. in
 $U$.
  \item[{\rm (d)}] the series $\displaystyle \sum_{k=1}^\infty G_k Fe_k$ converges  in
 $L^p(\Omega; U)$ for some $p\in [1,\infty)$.
  \item[{\rm (e)}] the series $\displaystyle \sum_{k=1}^\infty G_k Fe_k$ converges  in
 $L^p(\Omega; U)$ for all $p\in [1,\infty)$.
 \end{enumerate}
In this situation we have for every $p\in [1,\infty)$:
\begin{align*}
  \int_U \norm{u}^p\,\mu(du)=E\norm{\sum_{k=1}^\infty G_k Fe_k}^p.
\end{align*}
\end{theorem}
\begin{proof}
As in Remark \ref{re.repandimage} we obtain for the characteristic function of $\nu:=\gamma\circ F^{-1}$:
\begin{align*}
  \phi_{\nu}(u^\ast)
= \exp\left(-\tfrac{1}{2} \scapro{FF^\ast u^\ast}{u^\ast}\right)\qquad \text{for all }u^\ast\in U^\ast.
\end{align*}
This establishes the first equivalence between (a) and (b). The
proofs of the remaining part can be found in
\cite[Prop.4.2]{JanSeminar}.
\end{proof}
To show that $\gamma$-radonifying generalise Hilbert-Schmidt operators to Banach spaces we prove the result
by Mourier mentioned already above. Other proofs only relying on Hilbert theory can be found in the
literature.
\begin{corollary}\label{th.gammahilbert}
If $H$ and $U$ are separable Hilbert spaces then the following are
equivalent for $F\in L(H,U)$:
\begin{enumerate}
  \item[{\rm (a)}] $F$ is $\gamma$-radonifying;
  \item[{\rm (b)}] $F$ is Hilbert-Schmidt.
\end{enumerate}
\end{corollary}
\begin{proof}
Let $(e_k)_{k\in\N}$ be an orthonormal basis of $H$. The equivalence
follows immediately from
\begin{align*}
 E\norm{ \sum_{k=m}^n G_k Fe_k}^2=\sum_{k=m}^n \norm{Fe_k}^2
\end{align*}
for every family $(G_k)_{k\in\N}$ of independent standard normal
random variables.
\end{proof}

In general, the property of being $\gamma$-radonifying is not so easily accessible as Hilbert-Schmidt operators in case
of Hilbert spaces. However, for some specific Banach spaces, as $L^p$ or $l^p$ spaces, the set of all
covariance operators of Gaussian measures  can be also described more precisely, see \cite[Thm.V.5.5 and
Thm.V.5.6]{Vaketal}.

It turns out that the set of all $\gamma$-radonifying operators can be equipped with a norm such that it is a Banach space,
see \cite[Thm. 4.14]{JanSeminar}.


\section{Cylindrical stochastic processes}

Let $(\Omega, \A,P)$ be a probability space with a filtration
$\{\F_t\}_{t\ge 0}$.

Similarly to the correspondence between measures and random variables there is an analogue random
object associated to  cylindrical measures:
\begin{definition}\label{de.cylrv}
A {\em cylindrical random variable $X$ in $U$} is a linear map
\begin{align*}
 X:U^\ast \to L^0(\Omega).
\end{align*}
A cylindrical process $X$ in $U$ is a family $(X(t):\,t\ge 0)$ of cylindrical random
variables in $U$.
\end{definition}

The characteristic function of a cylindrical random  variable $X$ is defined by
\begin{align*}
 \phi_X:U^\ast\to\C, \qquad \phi_X(u^\ast)=E[\exp(i Xu^\ast)].
\end{align*}
The concepts of cylindrical measures and cylindrical random variables match perfectly. Because the
characteristic function of a cylindrical random variable is positive-definite and continuous on finite
subspaces there exists a cylindrical measure $\mu$ with the same characteristic function. We call $\mu$ the
{\em cylindrical distribution of $X$}. Vice versa, for every cylindrical measure $\mu$ on $\Cc(U)$ there
exists a probability space $(\Omega,\A,P)$ and a cylindrical random variable $X:U^\ast\to L^0(\Omega)$
such that  $\mu$ is the cylindrical distribution of $X$, see \cite[VI.3.2]{Vaketal}.

\begin{example}
 A cylindrical random variable $X:U^\ast \to L^0(\Omega)$ is called weakly Gaussian, if
 $Xu^\ast$ is Gaussian for all $u^\ast\in U^\ast$. Thus, $X$ defines a weakly Gaussian cylindrical measure
 $\mu$ on $\Cc(U)$. The characteristic function of $X$ coincide with the one of $\mu$ and is of the form
\begin{align*}
  \phi_X(u^\ast)=\exp(im(u^\ast)-\tfrac{1}{2}s(u^\ast))
\end{align*}
with $m:U^\ast\to \R$ linear and $s:U^\ast\to \Rp$ a quadratic form. If $X$ is strongly Gaussian there exists
a covariance operator $Q:U^\ast\to U$ such that
\begin{align*}
  \phi_X(u^\ast)=\exp(im(u^\ast)-\tfrac{1}{2}\scapro{Qu^\ast}{u^\ast}).
\end{align*}
Because $\phi_X(u^\ast)=\phi_{Xu^\ast}(1)$ it follows
\begin{align*}
  E[Xu^\ast]=m(u^\ast)\qquad\text{and}\qquad
  \Var[Xu^\ast]=\scapro{Qu^\ast}{u^\ast}.
\end{align*}
In the same way by comparing the characteristic function
\begin{align*}
\phi_{Xu^\ast,Xv^\ast}(\beta_1,\beta_2)
=E\left[\exp\left(i(\beta_1Xu^\ast +\beta_2Xv^\ast)\right)\right]
=E\left[\exp\left(i(X(\beta_1u^\ast +\beta_2v^\ast))\right)\right]
\end{align*}
for $\beta_1,\beta_2\in\R$ with the characteristic function of the
two-dimensional Gaussian vector $(Xu^\ast,Xv^\ast)$ we may conclude
\begin{align*}
 \text{Cov}[Xu^\ast, Xv^\ast]=\scapro{Qu^\ast}{v^\ast}.
\end{align*}
Let $H_Q$ denote the  reproducing kernel Hilbert space of the covariance operator $Q$. Then
we obtain
\begin{align*}
E\abs{Xu^\ast-m(u^\ast)}^2=\Var[Xu^\ast]=\scapro{Qu^\ast}{u^\ast}=
\norm{i_Q^\ast u^\ast}_{H_Q}^2.\\
\end{align*}
\end{example}

The cylindrical process $X=(X(t):\,t\ge 0)$ is called {\em adapted to a given filtration
$\{\F_t\}_{t\ge 0}$}, if $X(t)u^\ast$ is $\F_t$-measurable for all $t\ge 0$ and all $u^\ast\in
U^\ast$. The cylindrical process $X$ has weakly independent increments if for all $0\le
t_0<t_1<\dots <t_n$ and all $u^\ast_1,\dots, u^\ast_n\in U^\ast$ the random variables
\begin{align*}
(X(t_1)-X(t_0))u_1^\ast,\dots , (X(t_n)-X(t_{n-1}))u_n^\ast
\end{align*}
are independent.

\begin{remark}
Our definition of  cylindrical processes is based on the definitions
in \cite{BerRo83} and \cite{Vaketal}. In \cite{Metivier80} and
\cite{Schwartz96} cylindrical random variables are considered which
have values in $L^p(\Omega)$ for $p>0$. They assume in addition that
a cylindrical random variable is continuous. The continuity of a
cylindrical variable is reflected by continuity properties of its
characteristic function, see \cite[Prop.IV. 3.4]{Vaketal}. The
notion of weakly independent increments origins from \cite{BerRo83}.
\end{remark}

\begin{example}\label{ex.hat}
Let $Y=(Y(t):\,t\ge 0)$ be a stochastic process with values in a separable Banach space $U$.
Then $\hat{Y}(t)u^\ast:=\scapro{Y(t)}{u^\ast}$ for
$u^\ast\in U^\ast$ defines a cylindrical process $\hat{Y}=(\hat{Y}(t):\,t\ge 0)$. The cylindrical process
 $\hat{Y}$ is adapted if and only if  $Y$ is also adapted and  $\hat{Y}$ has weakly independent increments if and
 only if $Y$ has also independent increments. Both statements are due to the fact that the Borel and the cylindrical $\sigma$-algebras
coincide for separable Banach spaces due to Pettis' measurability theorem.
\end{example}

An {\em $\R^n$-valued Wiener process} $B=(B(t): \,t\ge 0)$ is an adapted stochastic
process
 with independent, stationary increments
$B(t)-B(s)$ which are normally distributed with expectation $E[B(t)-B(s)]=0$ and
covariance Cov$[B(t)-B(s),B(t)-B(s)]=\abs{t-s}C$ for a non-negative definite symmetric
matrix $C$. If $C=\Id$ we call $B$ a {\em standard} Wiener process.
\begin{definition}\label{de.weakcyl}
An adapted cylindrical process $W=(W(t):\,t\ge 0)$ in $U$ is a {\em
weakly cylindrical Wiener process}, if
\begin{enumerate}
\item[{\rm (a)}] for all $u^\ast_1,\dots, u^\ast_n\in U^\ast$ and $n\in \N$ the $\R^n$-valued
stochastic process
\begin{align*}
\big((W(t)u^\ast_1,\dots, W(t)u^\ast_{n}):\,t\ge 0\big)
\end{align*}
is a Wiener process.
\end{enumerate}
\end{definition}

Our definition of a weakly cylindrical Wiener process is an obvious
extension of the definition of a finite-dimensional Wiener processes
and is exactly in the spirit of cylindrical processes. The
multidimensional formulation in Definition \ref{de.weakcyl} would be
already necessary to define a finite-dimensional Wiener process by
this approach and it allows to conclude that a weakly cylindrical
Wiener process has weakly independent increments. The latter
property is exactly what is needed in addition to an one-dimensional
formulation:
\begin{lemma}\label{th.weaklyind}
For an adapted cylindrical process $W=(W(t):\,t\ge 0)$ the following are equivalent:
\begin{enumerate}
\item[{\rm (a)}] $W$ is a weakly cylindrical Wiener process;
\item[{\rm (b)}] $W$ satisfies
\begin{enumerate}
  \item[{\rm (i)}] $W$ has weakly independent increments;
  \item[{\rm (ii)}] $(W(t)u^\ast:\,t\ge 0)$ is a Wiener process for all $u^\ast\in U^\ast$.
\end{enumerate}
\end{enumerate}
\end{lemma}
\begin{proof} We have only to show that (b) implies (a). By linearity we have
\begin{align*}
  \beta_1(W(t)-W(s))u_1^\ast+\dots +\beta_n(W(t)-W(s))u_n^\ast
 = (W(t)-W(s))\left(\sum_{i=1}^n \beta_iu^\ast_i\right),
\end{align*}
for all $\beta_i\in\R$ and $u_i^\ast \in U^\ast$ which shows that
the increments of $((W(t)u^\ast_1,\dots, W(t)u^\ast_n)):\, t\ge 0)$
are normally distributed and stationary. The independence of the
increments follows by (i).
\end{proof}

Because $W(1)$ is a centred weakly Gaussian cylindrical random variable
 there exists a weakly Gaussian cylindrical measure $\mu$ such that
\begin{align*}
\phi_{W(1)}(u^\ast)= E[\exp(i W(1)u^\ast)] =
\phi_{\mu}(u^\ast)=\exp\left(-\tfrac{1}{2} s(u^\ast)\right)
\end{align*}
for a quadratic form $s:U^\ast\to \Rp$. Therefore, one obtains
\begin{align*}
\phi_{W(t)}(u^\ast)= E[\exp(iW(t)u^\ast)] =E[\exp\left(i
W(1)(tu^\ast)\right)] =\exp\left(-\tfrac{1}{2}t^2 s(u^\ast)\right)
\end{align*}
for all $t\ge 0$. Thus, the cylindrical distributions of $W(t)$ for all $t\ge 0$ are only determined by
the cylindrical distribution of $W(1)$.

\begin{definition}\label{de.strongW}
A weakly cylindrical Wiener process $(W(t):\, t\ge 0)$ is called
{\em strongly cylindrical Wiener  process}, if
\begin{enumerate}
  \item[{\rm (b)}] the cylindrical distribution $\mu$ of $W(1)$ is strongly Gaussian.
\end{enumerate}
\end{definition}

The additional condition on a weakly cylindrical Wiener  process to
be strongly requests the existence of an $U$-valued covariance
operator for the Gaussian cylindrical measure. To our knowledge
weakly cylindrical Wiener  processes are not defined  in the
literature and (strongly) cylindrical Wiener processes are defined
by means of other conditions. Often, these definition are formulated
by assuming the existence of the reproducing kernel Hilbert space.
But this implies the existence of the covariance operator. Another
popular way for defining cylindrical Wiener processes is by means of
a series. We will see in the next chapter that this is also
equivalent to our definition.

Later, we will compare a strongly cylindrical Wiener process with an
$U$-valued Wiener process. Also the latter is defined as a direct
generalisation of a real-valued Wiener process:
\begin{definition}
An adapted $U$-valued stochastic process $(W(t):\,t\ge 0)$ is called a {\em Wiener process }
if
\begin{enumerate}
\item[{\em (a)}] $W(0)=0$ $P$-a.s.;
\item[{\em (b)}] $W$ has independent, stationary increments;
\item[{\em (c)}] there exists a Gaussian covariance operator $Q:U^\ast\to U$ such that
\begin{align*}
 W(t)-W(s)\stackrel{d}{=} N(0,(t-s)Q)\qquad\text{for all }0\le s\le
 t.
\end{align*}
 \end{enumerate}
\end{definition}

If $U$ is finite dimensional then $Q$ can be any symmetric, positive semi-definite matrix. In case that $U$
is a Hilbert space we know already that $Q$ has to be nuclear. For the general case of a Banach space $U$ we
can describe the possible Gaussian covariance operator by Theorem \ref{th.eqradonifying}.

It is obvious that every $U$-valued Wiener process $W$ defines a
strongly cylindrical Wiener process $(\hat{W}(s):\,t\ge 0)$ in $U$
by $\hat{W}(s)u^\ast:=\scapro{W(s)}{u^\ast}$. For the converse
question, if a cylindrical Wiener process can be represented in such
a way by an $U$-valued Wiener process we will derive later necessary
and sufficient conditions.

\section{Representations of cylindrical Wiener processes}

In this section we derive representations of cylindrical Wiener processes and $U$-valued Wiener processes in
terms of some series. In addition, these representations can also serve as a construction of these processes,
see Remark \ref{re.construction}.

\begin{theorem}\label{th.cylsum}
For an adapted cylindrical process  $W:=(W(t):\,t\ge 0)$ the following are equivalent:
\begin{enumerate}
\item[{\rm (a)}] $W$ is a strongly cylindrical Wiener process;
\item[{\rm (b)}] there exist a Hilbert space $H$ with an orthonormal basis $(e_n)_{n\in\N}$, $F\in L(H,U)$
and independent real-valued standard Wiener processes $(B_n)_{n\in\N}$ such that
\begin{align*}
 W(t)u^\ast=\sum_{k=1}^\infty \scapro{Fe_k}{u^\ast} B_k(t) \qquad
    \text{in }L^2(\Omega)\text{ for all $u^\ast\in U^\ast$}.
\end{align*}
\end{enumerate}
\end{theorem}
\begin{proof}
(b) $\;\Rightarrow\;$ (a) By Doob's inequality we obtain for any $m,n\in\N$
\begin{align*}
 E\left[ \sup_{t\in [0,T]}\abs{\sum_{k=n}^{n+m} \scapro{Fe_k}{u^\ast} B_k(t)}^2\right]
&\le 4 E \abs{\sum_{k=n}^{n+m} \scapro{Fe_k}{u^\ast} B_k(T)}^2\\
&= 4 T\sum_{k=n}^{n+m} \scapro{e_k}{F^\ast u^\ast}^2\\
&\to 0 \qquad\text{for }m,n\to\infty.
\end{align*}
Thus, for every $u^\ast\in U^\ast$ the random variables $W(t)u^\ast$
are well defined and form a cylindrical process $(W(t):\,t\ge 0)$.
For any  $0=t_0<t_1<\dots <t_m$ and $\beta_k\in\R$ we calculate
\begin{align*}
& E\left[\exp\left(i\sum_{k=0}^{m-1} \beta_k (W(t_{k+1})u^\ast - W(t_k)u^\ast)\right)\right]\\
 &\qquad=\lim_{n\to\infty}
  E\left[\exp\left(i\sum_{k=0}^{m-1} \beta_k \sum_{l=1}^n \scapro{Fe_l}{u^\ast}(B_l(t_{k+1})-B_l(t_k))
\right)\right]\\
&\qquad =\lim_{n\to\infty} \prod_{k=0}^{m-1}\prod_{l=1}^n E\Big[ \exp \left( i\beta_k
  \scapro{Fe_l}{u^\ast}(B_l(t_{k+1})-B_l(t_k))\right)\Big]\\
&\qquad =\lim_{n\to\infty} \prod_{k=0}^{m-1}\prod_{l=1}^n  \exp \Big( -\tfrac{1}{2}\beta_k^2
  \scapro{Fe_l}{u^\ast}^2(t_{k+1}-t_k)\Big)\\
&\qquad = \prod_{k=0}^{m-1} \exp \left( -\tfrac{1}{2}\beta_k^2
  \norm{F^\ast u^\ast}^2_H(t_{k+1}-t_k)\right),
\end{align*}
which shows that $(W(t)u^\ast:\,t\ge 0)$ has independent, stationary Gaussian increments. Because the partial
sums converge uniformly on every finite interval the process $(W(t)u^\ast:\,t\ge 0)$ has  a.s. continuous
paths and is therefore established as a real-valued Wiener process.

The calculation above  of the characteristic function yields
\begin{align*}
  E\left[\exp(i W(1)u^\ast)\right]=\exp\left(-\tfrac{1}{2} \norm{F^\ast u^\ast}_H^2\right)
  =\exp\left(-\tfrac{1}{2}\scapro{FF^\ast u^\ast}{u^\ast}^2\right).
\end{align*}
Hence, the process $W$ is a strongly cylindrical Wiener process with covariance operator $Q:=FF^\ast$.

(a) $\Rightarrow$ (b): Let $Q:U^\ast\to U$ be the covariance
operator of $W(1)$ and $H$ its reproducing kernel Hilbert space with
the inclusion mapping $i_Q:H\to U$. Because the range of $i_Q^\ast$
is dense in $H$ and $H$ is separable there exists an orthonormal
basis $(e_n)_{n\in\N}\subseteq$range$(i_Q^\ast)$ of $H$. We choose
$u_n^\ast\in U^\ast$ such that $i_Q^\ast u_n^\ast=e_n$ for all
$n\in\N$ and define $B_n(t):=W(t)u_n^\ast$. Then we obtain
\begin{align*}
  E\abs{ \sum_{k=1}^n \scapro{i_Qe_k}{u^\ast}B_k(t) - W(t)u^\ast }^2
  &=E\left[W(t)\left(\sum_{k=1}^n \scapro{i_Qe_k}{u^\ast} u_k^\ast -u^\ast\right)\right]^2\\
 &=t\norm{i_Q^\ast\left(\sum_{k=1}^n \scapro{i_Qe_k}{u^\ast}u_k^\ast -u^\ast\right)}^2_H\\
 &=t\norm{\sum_{k=1}^n [e_k,i_Q^\ast u^\ast]_{H_Q}e_k -i_Q^\ast u^\ast }^2_H\\
 &\to 0 \qquad\text{for }n\to\infty.
\end{align*}
Thus, $W$ has the required representation and it remains to establish that the Wiener processes
$B_n:=(B_n(t):\,t\ge 0)$ are independent. Because of the Gaussian distribution it is sufficient to establish
that $B_n(s)$ and $B_m(t)$ for any $s\le t$ and $m,n\in\N$ are independent:
\begin{align*}
 E[B_n(s)B_m(t)]&=
 E[W(s)u_n^\ast W(t)u_m^\ast]\\
& = E[W(s)u^\ast_n (W(t)u_m^\ast- W(s)u_m^\ast)] + E[W(s)u_n^\ast
W(s)u_m^\ast]. \intertext{The first term is zero by Theorem
\ref{th.weaklyind} and for the second term we obtain}
 E[W(s)u_n^\ast W(s)u_m^\ast]
 &=s\scapro{Qu_n^\ast}{u_m^\ast}
 = s [i_Q^\ast u_n^\ast,i_Q^\ast u_m^\ast]_{H_Q}
 =s  [e_n,e_m]_{H_Q} =s \delta_{n,m}.
\end{align*}
Hence, $B_n(s)$ and $B_m(t)$ are uncorrelated and therefore independent.
\end{proof}

\begin{remark}
  The proof has shown that the Hilbert space $H$ in part (b) can be chosen as the reproducing kernel
  Hilbert space associated to the Gaussian cylindrical distribution of $W(1)$. In this
  case the function $F:H \to U$ is the inclusion mapping $i_Q$.
\end{remark}

\begin{remark}
Let $H$ be  a separable Hilbert space with orthonormal basis
  $(e_k)_{k\in\N}$ and $(B_k(t):\,t\ge 0)$ be independent real-valued Wiener
  processes. By setting $U=H$ and $F=\Id$ Theorem \ref{th.cylsum} yields that a strongly cylindrical
  Wiener process $(W_H(t):\, t\ge 0)$ is defined by
  \begin{align*}
     W_H(t)h= \sum_{k=1}^\infty \scapro{e_k}{h} B_k(t)
  \qquad\text{for all }h\in H .
  \end{align*}
The covariance operator of $W_H$ is $\Id:H\to H$. This is the
approach how a cylindrical Wiener process is defined for example in
\cite{Bogachev98} and \cite{NeeWeis}.

If in addition $U$ is a separable Banach space and $F\in L(H,U)$ we
obtain by defining
\begin{align*}
W(t)u^\ast := W_H(t)(F^\ast u^\ast) \qquad\text{for all }u^\ast\in U^\ast,
\end{align*}
a strongly cylindrical Wiener process $(W(t):\, t\ge 0)$ with covariance operator
$Q:=FF^\ast$ according to our Definition \ref{de.strongW}.
\end{remark}


\begin{theorem}\label{th.Usum}
For an adapted $U$-valued process  $W:=(W(t):\,t\ge 0)$ the following are equivalent:
\begin{enumerate}
\item[{\rm (a)}] $W$ is an $U$-valued Wiener process;
\item[{\rm (b)}] there exist a Hilbert space $H$ with an orthonormal basis $(e_n)_{n\in\N}$, a $\gamma$-radonifying
operator $F\in L(H,U)$ and independent real-valued standard Wiener processes $(B_n)_{n\in\N}$ such that
\begin{align*}
 W(t)=\sum_{k=1}^\infty Fe_k B_k(t) \qquad
    \text{in }L^2(\Omega;U).
\end{align*}
\end{enumerate}
\end{theorem}
\begin{proof} (b) $\Rightarrow$ (a): As in the proof of Theorem \ref{th.cylsum}
 we obtain by Doob's Theorem  (but here for infinite-dimensional spaces) for
any $m,n\in \N$
\begin{align*}
 E\left[ \sup_{t\in [0,T]}\norm{\sum_{k=n}^{n+m} Fe_k B_k(t)}^2\right]
&\le 4 E \norm{\sum_{k=n}^{n+m} Fe_k  B_k(T)}^2\\
&\to 0 \qquad\text{for }m,n\to\infty,
\end{align*}
where the convergence follows by Theorem \ref{th.eqradonifying} because $F$ is $\gamma$-radonifying. Thus,
the random variables $W(t)$ are well defined and form an $U$-valued stochastic process $W:=(W(t):\,t\ge 0)$.
As in the proof of Theorem \ref{th.cylsum} we can proceed to establish that $W$ is an $U$-valued Wiener
process.

(a) $\Rightarrow$ (b): By Theorem \ref{th.cylsum} there exist a
Hilbert space $H$ with an orthonormal basis $(e_n)_{n\in\N}$, $F\in
L(H,U)$ and independent real-valued standard Wiener processes
$(B_n)_{n\in\N}$ such that
\begin{align*}
 \scapro{W(t)}{u^\ast}=\sum_{k=1}^\infty \scapro{Fe_k}{u^\ast} B_k(t) \qquad
    \text{in }L^2(\Omega)\text{ for all $u^\ast\in U^\ast$}.
\end{align*}
The It{\^o}-Nisio Theorem \cite[Thm.V.2.4]{Vaketal} implies
\begin{align*}
W(t)=\sum_{k=1}^\infty Fe_k B_k(t) \qquad
    \text{$P$-a.s. for all $u^\ast\in U^\ast$}
\end{align*}
and a result by Hoffmann-Jorgensen \cite[Cor.2 in V.3.3]{Vaketal}
yields the convergence in $L^2(\Omega;U)$. Theorem
\ref{th.eqradonifying} verifies $F$ as $\gamma$-radonifying.
\end{proof}

\begin{remark}\label{re.construction}
 In the proofs of the implication from (b) to (a) we established in  both Theorems \ref{th.cylsum} and \ref{th.Usum}
 even more than required: we established the convergence of the series in the specified sense without
 assuming the existence of the limit process, respectively. This means, that we can read these results also
 as a construction principle of cylindrical or $U$-valued Wiener processes without assuming the existence of the considered
 process a priori.

 The construction of these random objects differs significantly in the required conditions on the involved operator $F$.
 For a cylindrical Wiener process no conditions are required, however, for an $U$-valued Wiener process we have to guarantee
 $Q=FF^\ast$ to be a covariance operator of a Gaussian measure by assuming $F$ to be $\gamma$-radonifying.
\end{remark}

\section{When is  a  cylindrical Wiener process $U$-valued ?}

In this section we give equivalent conditions for a strongly cylindrical Wiener process
to be an $U$-valued Wiener process.  To be more precise a cylindrical random variable
$X:U^\ast \to L^0(\Omega)$ is called {\em induced by a random variable $Z:\Omega\to U$},
if $P$-a.s.
\begin{align*}
 Xu^\ast= \scapro{Z}{u^\ast}\qquad\text{for all }u^\ast\in U^\ast.
\end{align*}
This definition generalises in an obvious way to cylindrical processes.

Because of the correspondence to cylindrical measures the question whether a cylindrical random variable is
induced by an $U$-valued random variable is reduced to the question whether the cylindrical measure extends
to a Radon measure (\cite[Thm. Vi.3.1]{Vaketal}). There is a classical answer by Prokhorov (\cite[Thm.
VI.3.2]{Vaketal}) to this question in terms of tightness. A cylindrical measure $\mu$ on $\Cc(U)$ is called
{\em tight} if for each $\epsilon>0$ there exists  a compact subset $K\subseteq U$ such that
 \begin{align*}
  \mu(K)\ge 1-\epsilon.
 \end{align*}
In case of non-separable Banach spaces $U$ one has to be more careful because then compact sets are not
necessarily admissible arguments of a cylindrical measure.

\begin{theorem}\label{th.cyl=radon}
For a strongly cylindrical Wiener process  $W:=(W(t):\,t\ge 0)$ with covariance operator
$Q=i_Qi_Q^\ast$ the following are equivalent:
\begin{enumerate}
\item[{\rm (a)}] $W$ is induced by an $U$-valued Wiener process;
\item[{\rm (b)}] $i_Q$ is $\gamma$-radonifying;
\item[{\rm (c)}] the cylindrical distribution of $W(1)$ is tight;
\item[{\rm (d)}] the cylindrical distribution of $W(1)$ extends to a measure.
\end{enumerate}
\end{theorem}
\begin{proof}
(a) $\Rightarrow$ (b) If there exists an $U$-valued Wiener process $(\tilde{W}(t):\,t\ge 0)$ with
$W(t)u^\ast=\scapro{\tilde{W}(t)}{u^\ast}$ for all $u^\ast\in U^\ast$, then $\tilde{W}(1)$ has a Gaussian
distribution with covariance operator $Q$. Thus, $i_Q$ is $\gamma$-radonifying by Theorem
\ref{th.eqradonifying}.

(b)$\Leftrightarrow$ (c) $\Leftrightarrow$ (d) This is Prokhorov's Theorem on cylindrical measures.

(b)$\Rightarrow $(a) Due to Theorem \ref{th.cylsum} there exist an orthonormal basis $(e_n)_{n\in\N}$ of the
reproducing kernel Hilbert space of $Q$ and independent standard real-valued Wiener process $(B_k(t):\,t\ge
0)$ such that
\begin{align*}
W(t)u^\ast=\sum_{k=1}^\infty \scapro{i_Q e_k}{u^\ast} B_k(t)\qquad\text{for all } u^\ast\in U^\ast.
\end{align*}
On the other hand, because $i_Q$ is $\gamma$-radonifying Theorem \ref{th.Usum} yields that
\begin{align*}
\tilde{W}(t)=\sum_{k=1}^\infty i_Q e_k B_k(t)
\end{align*}
defines an $U$-valued Wiener process $(\tilde{W}(t):\,t\ge 0)$. Obviously, we have
 $W(t)u^\ast=\scapro{\tilde{W}(t)}{u^\ast}$ for all $u^\ast$.
\end{proof}

If $U$ is a separable Hilbert space we can replace the condition (b) by
\begin{enumerate}
\item[(b$^\prime$)] $i_Q$ is Hilbert-Schmidt
\end{enumerate}
because of Theorem \ref{th.gammahilbert}.

\section{Integration}

In this section we introduce an integral with respect to a strongly cylindrical Wiener
process $W=(W(t):\,t\ge 0)$ in $U$. The integrand is a stochastic process with values in
$L(U,V)$, the set of bounded linear operators from $U$ to $V$, where $V$ denotes a
separable Banach space. For that purpose we assume for $W$ the representation according
to Theorem \ref{th.cylsum}:
\begin{align*}
 W(t)u^\ast=\sum_{k=1}^\infty \scapro{i_Qe_k}{u^\ast} B_k(t) \qquad
    \text{in }L^2(\Omega)\text{ for all $u^\ast\in U^\ast$},
\end{align*}
where $H$ is the reproducing kernel Hilbert space of the covariance operator $Q$ with the
inclusion mapping $i_Q:H\to U$ and an orthonormal basis  $(e_n)_{n\in\N}$ of $H$. The
real-valued standard Wiener processes $(B_k(t):\,t\ge 0)$ are defined by
$B_k(t)=W(t)u_k^\ast$  for some $u_k^\ast\in U^\ast$ with $i_Q^\ast u_k^\ast=e_k$.
\begin{definition}
  The set $M_T(U,V)$ contains all random variables $\Phi:[0,T]\times \Omega\to L(U,V)$ such that:
\begin{enumerate}
  \item[{\rm (a)}] $(t,\omega)\mapsto \Phi^\ast(t,\omega)v^\ast$ is $\Borel[0,T]\otimes \A$ measurable for all
             $v^\ast\in V^\ast$;
  \item[{\rm (b)}] $\omega\mapsto \Phi^\ast(t,\omega)v^\ast$ is $\F_t$-measurable for all
       $v^\ast\in V^\ast$ and $t\in [0,T]$;
  \item[{\rm (c)}] $\displaystyle \int_0^T E\norm{\Phi^\ast(s,\cdot)v^\ast}_{U^\ast}^2\,ds<\infty \;$ for all $v^\ast\in
  V^\ast$.
\end{enumerate}
\end{definition}
As usual we neglect  the dependence of  $\Phi\in M_T(U,V)$ on $\omega$
and write $\Phi(s)$ for $\Phi(s,\cdot)$ as well as for the dual
operator $\Phi^\ast(s):=\Phi^\ast(s,\cdot)$ where
$\Phi^\ast(s,\omega)$ denotes the dual operator of
$\Phi(s,\omega)\in L(U,V)$.

We define the candidate for a stochastic integral:
\begin{definition}\label{de.I_t}
For $\Phi\in M_T(U,V)$ we define
\begin{align*}
 I_t(\Phi)v^\ast:= \sum_{k=1}^\infty \int_0^t \scapro{\Phi(s)i_Q e_k}{v^\ast}\, dB_k(s)
 \qquad \text{in }L^2(\Omega)
\end{align*}
for all $v^\ast\in V^\ast$ and $t \in [0,T]$.
\end{definition}

The stochastic integrals appearing in Definition \ref{de.I_t} are the  known real-valued
It{\^o} integrals and they are well defined thanks to our assumption on $\Phi$. In the
next Lemma we establish that the asserted limit exists:

\begin{lemma}\label{le.cylintwell}
$I_t(\Phi):V^\ast \to L^2(\Omega)$ is a well-defined cylindrical random variable in $V$ which is independent
of the representation of $W$, i.e. of $(e_n)_{n\in\N}$ and $(u_n^\ast)_{n\in\N}$.
\end{lemma}
\begin{proof}
We begin to establish the convergence in $L^2(\Omega)$. For
that, let $m,n\in \N$ and we define for simplicity
$h(s):=i_Q^\ast\Phi^\ast(s)v^\ast$. Doob's theorem implies
 \begin{align*}
& E\abs{\sup_{0\le t\le T} \sum_{k=m+1}^{n}\int_0^t \scapro{\Phi(s)i_Q e_k}{v^\ast}\,dB_k(s)}^2\\
&\qquad \le 4 \sum_{k=m+1}^{n}\int_0^T E\scaproh{e_k}{h(s)}^2\,ds\\
&\qquad \le 4 \sum_{k=m+1}^{\infty}\int_0^T E\scaproh{\scaproh{e_k}{h(s)}e_k}{h(s)}\,ds\\
&\qquad = 4 \sum_{k=m+1}^{\infty}\sum_{l=m+1}^\infty\int_0^T E\scaproh{\scaproh{e_k}{h(s)}e_k}{\scaproh{e_l}{h(s)}e_l}\,ds\\
&\qquad =4 \int_0^T E\norm{(\Id -\pi_m)h(s)}_H^2\,ds,
 \end{align*}
where $\pi_m:H\to H$ denotes the projection onto the span of $\{e_1,\dots, e_m\}$. Because $\norm{(\Id
-\pi_m)h(s)}_H^2\to 0$ for $m\to \infty$ and
\begin{align*}
  \int_0^T E\norm{(\Id -\pi_m)h(s)}_H^2\,ds
  \le \norm{i_Q^\ast}^2_{U^\ast\to H} \int_0^T E\norm{\Phi^\ast(s,\cdot)v^\ast}^2_{U^\ast}\,ds<\infty
\end{align*}
we obtain by Lebesgue's theorem the convergence in $L^2(\Omega)$.

To prove the independence on the chosen representation of $W$ let $(f_l)_{l\in\N}$ be an
other orthonormal basis of $H$ and $w^\ast_l \in U^\ast $ such that $i_Q^\ast
w^\ast_l=f_l$ and $(C_l(t):\,t\ge 0)$ independent Wiener processes defined by
$C_l(t)=W(t)w^\ast_l$. As before we define in $L^2(\Omega)$:
\begin{align*}
  \tilde{I}_t(\Phi)v^\ast:=\sum_{l=1}^\infty \int_0^t \scapro{\Phi(s)i_Q f_l}{v^\ast}\,dC_l(s)
  \qquad\text{for all }v^\ast\in V^\ast.
\end{align*}
The relation $\Cov(B_k(t), C_l(t))=t \scaproh{i_Q^\ast u_k^\ast}{i_Q^\ast w_l^\ast}= t \scaproh{e_k}{f_l}$
enables us to calculate
\begin{align*}
& E\abs{I_t(\Phi)v^\ast-\tilde{I}_t(\Phi)v^\ast}^2\\
&\qquad= E\abs{I_t(\Phi)v^\ast}^2+ E\abs{\tilde{I}_t(\Phi)v^\ast}^2
  -2 E\left[ \big(I_t(\Phi)v^\ast\big)\big( \tilde{I}_t(\Phi)v^\ast\big)\right]\\
&\qquad= \sum_{k=1}^\infty \int_0^t E\scapro{\Phi(s)i_Qe_k}{v^\ast}^2\,ds
 + \sum_{l=1}^\infty \int_0^t E\scapro{\Phi(s)i_Qf_l}{v^\ast}^2\,ds\\
&\qquad\qquad -2 \sum_{k=1}^\infty \sum_{l=1}^\infty \int_0^t E\left[\scapro{\Phi(s)i_Q e_k}{v^\ast}
  \scapro{\Phi(s)i_Q f_l}{v^\ast} \scaproh{i_Q^\ast u_k^\ast}{i_Q^\ast w_l^\ast}\right]\,ds\\
& \qquad = 2\int_0^t E\norm{i_Q^\ast\Phi^\ast(s) v^\ast}_H^2\,ds
 -2 \int_0^t E\norm{i_Q^\ast\Phi^\ast (s) v^\ast}_H^2\,ds \\
&\qquad =0,
\end{align*}
which proves the independence of $I_t(\Phi)$ on
$(e_k)_{k\in\N}$ and $(u_k^\ast)_{k\in\N}$.

The linearity of $I_t(\Phi)$ is obvious and hence the proof is complete.
\end{proof}

Our next definition is not very surprising:
\begin{definition}
  For $\Phi\in M_T(U,V)$ we call the cylindrical random variable
  \begin{align*}
    \int_0^t \Phi(s)\, dW(s):=I_t(\Phi)
  \end{align*}
{\em cylindrical stochastic integral with respect to $W$}.
\end{definition}

Because the cylindrical stochastic integral is strongly based on the well-known
real-valued It{\^o} integral many features can be derived easily. We collect the
martingale property and It{\^o}'s isometry in the following theorem.
\begin{theorem}
Let $\Phi$ be in $M_T(U,V)$. Then we have
\begin{enumerate}
\item[{\rm (a)}] for every $v^\ast\in V^\ast$ the family
\begin{align*}
 \left(\Big(\int_0^t \Phi(s)\,dW(s)\Big)v^\ast:\,t\in [0,T]\right)
\end{align*}
forms a continuous  square-integrable martingale.
\item[{\rm (b)}] the It{\^o}'s isometry:
\begin{align*}
  E\abs{\Big(\int_0^t \Phi(s)\,dW(s)\Big)v^\ast}^2=
  \int_0^t E\norm{i_Q^\ast\Phi^\ast(s) v^\ast}_H^2\,ds.
\end{align*}
\end{enumerate}
\end{theorem}

\begin{proof} (a)
  In Lemma \ref{le.cylintwell} we have identified $I_t(\Phi)v^\ast$ as the limit of
\begin{align*}
  M_n(t):= \sum_{k=1}^{n}\int_0^t \scapro{\Phi(s)i_Q e_k}{v^\ast}\,dB_k(s),
\end{align*}
where the convergence takes place in $L^2(\Omega)$ uniformly on the interval $[0,T]$. As
$(M_n(t):\,t\in [0,T])$ are continuous martingales the assertion follows.

(b) Using It{\^o}'s isometry for real-valued stochastic integrals we obtain
\begin{align*}
 E\abs{\Big(\int_0^t \Phi(s)\,dW(s)\Big)v^\ast}^2
  &= \sum_{k=1}^\infty  E\left[\int_0^T \scapro{\Phi(s)i_Q e_k}{v^\ast}\,dB_k(s)\right]^2\\
 &=\sum_{k=1}^\infty \int_0^T  E\scaproh{e_k}{i_Q^\ast \Phi^\ast (s) v^\ast}^2\,ds\\
 &= \int_0^T E\norm{i_Q^\ast\Phi^\ast (s) v^\ast}_H^2\, ds.
\end{align*}
\end{proof}

An obvious question is under which conditions the cylindrical integral is induced by a
$V$-valued random variable. The answer to this question will also allow us to relate the
cylindrical integral with other known definitions of stochastic integrals in infinite
dimensional spaces.

From our point of view the following corollary is an obvious consequence. We call an
stochastic process $\Phi\in M_T(U,V)$ non-random if it does not depend on
$\omega\in\Omega$.
\begin{corollary}\label{co.intind}
For non-random $\Phi\in M_T(U,V)$ the following are equivalent:
\begin{enumerate}
\item[{\rm (a)}] $\displaystyle \int_0^T \Phi(s)\,dW(s)$ is induced by a $V$-valued random variable;
\item[{\rm (b)}] there exists a Gaussian measure $\mu$ on $V$ with covariance operator $R$ such that:
\begin{align*}
  \int_0^T \norm{i_Q^\ast \Phi^\ast (s)v^\ast}_H^2\,ds =\scapro{Rv^\ast}{v^\ast}
  \qquad\text{for all }v^\ast\in V^\ast.
\end{align*}
\end{enumerate}
\end{corollary}
\begin{proof}
(a) $\Rightarrow$ (b): If the integral $I_T(\Phi)$ is induced by a $V$-valued random
variable then the random variable is centred Gaussian, say with a covariance operator
$R$. Then It{\^o}'s isometry yields
\begin{align*}
  \scapro{Rv^\ast}{v^\ast}
=E\abs{I_T(\Phi)v^\ast}^2 = \int_0^T \norm{i_Q^\ast\Phi^\ast (s) v^\ast}_H^2\, ds.
\end{align*}

(b)$\Rightarrow $(a): Again It{\^o}'s isometry  shows that the weakly Gausian cylindrical
distribution of $I_T(\Phi)$ has covariance operator $R$ and thus, extends to a Gaussian
measure on $V$.
\end{proof}

The condition (b) of Corollary \ref{co.intind} is derived in van Neerven and Weis
\cite{NeeWeis} as a sufficient and necessary condition for the existence of the
stochastic Pettis integral introduced in this work. Consequently, it is easy to see that
under the equivalent conditions (a) or (b) the cylindrical integral coincides with this
stochastic Pettis integral. Further relation of condition (b) to $\gamma$-radonifying
properties of the integrand $\Phi$ can also be found in  \cite{NeeWeis}.

Our next result relates the cylindrical integral to the stochastic integral in Hilbert
spaces as introduced in Da Prato and Zabczyk \cite{DaPrato92}. For that purpose, we
assume that $U$ and $V$ are separable Hilbert spaces. Let $W$ be a strongly cylindrical
Wiener process in $U$ and let the inclusion mapping $i_Q:H_Q\to U$ be Hilbert-Schmidt.
Then there exist an orthonormal basis $(f_k)_{k\in \N}$ in $U$ and real numbers
$\lambda_k\ge 0$ such that $Qf_k=\lambda_kf_k$ for all $k\in \N$. For the following we
can assume that $\lambda_k\neq 0$ for all $k\in\N$. By defining
$e_k:=\sqrt{\lambda_k}f_k$ for all $k\in \N$ we obtain an orthonormal basis of $H_Q$ and
$W$ can be represented as usual as a sum with respect to this orthonormal basis.

Our assumption on $i_Q$ to be Hilbert-Schmidt is not a restriction because in general the
integral with respect to a strongly cylindrical Wiener  process is defined in
\cite{DaPrato92} by extending $U$ such that $i_Q$ becomes Hilbert-Schmidt.

\begin{corollary}
 Let $W$ be a strongly cylindrical Wiener process in a separable Hilbert space $U$ with $i_Q:H_Q\to U$ Hilbert-Schmidt.
 If $V$ is a separable Hilbert space and $\Phi\in M_T(U,V)$ is such that
 \begin{align*}
   \sum_{k=1}^\infty \lambda_k\ \int_0^T E\norm{\Phi(s)i_Qe_k}^2_V\,ds <\infty,
 \end{align*}
then the cylindrical integral
\begin{align*}
  \int_0^T \Phi(s)\,dW(s)
\end{align*}
is induced by a $V$-valued random variable. This random variable is the standard
stochastic integral in Hilbert spaces of $\Phi$ with respect to $W$.
\end{corollary}

\begin{proof}
By Theorem \ref{th.cyl=radon} the cylindrical Wiener process $W$ is induced by an $U$-valued Wiener process $Y$.  We define $U$-valued Wiener processes
$(Y_N(t):\,t\in [0,T])$ by
  \begin{align*}
    Y_N(t)=\sum_{k=1}^N i_Q e_k B_k(t).
  \end{align*}
Theorem \ref{th.Usum} implies that $Y_N(t)$ converges to $Y$ in $L^2(\Omega;U)$. By our
assumption on $\Phi$ the stochastic integrals $\Phi\circ Y_N(T)$ in the sense of Da Prate
and Zabczyk \cite{DaPrato92} exist and converge to the stochastic integral $\Phi\circ
Y(T)$ in $L^2(\Omega;V)$, see \cite[Ch.4.3.2]{DaPrato92}.

On the other hand, by first considering simple functions $\Phi$ and then extending to the
general case we obtain
\begin{align*}
  \scapro{\Phi\circ Y_N(T)}{v^\ast}=\sum_{k=1}^N \int_0^t \scapro{\Phi(s)i_Qe_k}{v^\ast}\, dB_k(s)
\end{align*}
for all $v^\ast\in V^\ast$. By Definition \ref{de.I_t} the right hand side converges in
$L^2(\Omega)$ to
\begin{align*}
\left(\int_0^T \Phi(s)\,dW(s)\right)v^\ast,
\end{align*}
whereas at least a subsequence of $(\scapro{\Phi\circ Y_N(T)}{v^\ast})_{N\in\N}$
converges to $\scapro{\Phi\circ Y(T)}{v^\ast}$ $P$-a.s..
\end{proof}

Based on the cylindrical integral one can build up a whole theory of {\em cylindrical
stochastic differential equations}. Of course, a solution will be in general a
cylindrical process but there is no need to put geometric constrains on the state space
under consideration. If one is interested in classical stochastic processes as solutions
for some reasons one can tackle this problem as in our two last results by deriving
sufficient conditions guaranteeing that the cylindrical solution is induced by a
$V$-valued random process.

\end{document}